\documentclass[final,twoside,11pt]{entics} 
\usepackage{enticsmacro}
\usepackage{graphicx}
\usepackage[all]{xy}
\usepackage{hyperref}
\usepackage{caption}
\usepackage{amsmath}
\newcommand{\Spec}{\operatorname{Spec}}
\newcommand{\Max}{\operatorname{Max}}
\newcommand{\BL}{\operatorname{BL}}
\newcommand{\MTL}{\operatorname{MTL}}
\newcommand{\NM}{\operatorname{NM}}
\newcommand{\MV}{\operatorname{MV}}
\newcommand{\local}{\operatorname{local}}

\def\conf{ISDT 2022} 	
\volume{2}			
\def\volu{2}			
\def\papno{15}			
\def\mydoi{{\normalshape  \href{https://doi.org/10.46298/entics.10399}{10.46298/entics.10399}}}
\def\copynames{H. Zhang, D. Zhao} 
\def\runtitle{The sheaf representation of residuated lattices}

\def\lastname{Zhang and Zhao}
\begin{document}
\begin{frontmatter}
  \title{The Sheaf Representation of Residuated Lattices} 
  \author{Huarong Zhang\thanksref{a}\thanksref{ALL}\thanksref{myemail}}	
   \author{Dongsheng Zhao\thanksref{b}\thanksref{coemail}}		
   \address[a]{College of Sciences\\ China Jiliang University\\				
    Hangzhou, China}  							
  \thanks[ALL]{I would like to express my sincere thanks to Professor Dongsheng Zhao who hosted my visit to National Institute of Education, Nanyang Technological University. The authors  would like to thank the anonymous reviewers for their professinal comments that have improved this paper substantially. This work is supported by the National Natural Science Foundation of China (No. 11701540, 12171445)}   
   \thanks[myemail]{Email: \href{mailto:hrzhang2008@cjlu.edu.cn} {\texttt{\normalshape hrzhang2008@cjlu.edu.cn}}} 
  \address[b]{Mathematics and Mathematics Education\\National Institute of Education\\Nanyang Techonological University\\ Singapore 637616} 
  \thanks[coemail]{Email: \href{mailto: dongsheng.zhao@nie.edu.sg} {\texttt{\normalshape dongsheng.zhao@nie.edu.sg}}}
\begin{abstract} 
 The residuated lattices form one of the most important algebras of fuzzy logics and have been heavily studied by people from various different points of view. Sheaf presentations provide a topological approach to many algebraic structures.
 In this paper, we study the topological properties of prime spectrum of residuated lattices, and then construct a sheaf space to obtain a sheaf representation for each residuated lattice.
\end{abstract}
\end{frontmatter}
\section{Introduction}\label{intro}
Residuated lattices, introduced by Ward and Dilworth \cite{wd}, constitute
the semantics of H\"{o}hle's Monoidal Logic \cite{hk}. Such algebras provide the fundamental framework for algebras of logics. Many familiar algebras, such as Boolean algebras, $\MV$-algebras, $\BL$-algebras, $\MTL$-algebras, $\NM$ algebras ($R_0$-algebras) and Heyting algebras, are special types of residuated lattices.

In dealing with certain type of problems, sheaf representations of algebras often provide powerful tools as they convert the study of algebras to the study of stalks, a topological structure. Thus, in the past decades, sheaf spaces \cite{d,mm,te} have been constructed for various types of algebras to obtain their corresponding sheaf representations.

In the case of algebras for fuzzy logics, Ghilardi and Zawadowski constructed the Grothendieck-type duality and got the sheaf representation for Heyting algebras \cite{gh}. Many scholars investigated the sheaf representations of $\MV$-algebras \cite{ne,dp,fl,fg,mg,gg}. Here we give an outline of their differences. In \cite{dp}, Dubuc and Poveda use $\Spec(L)$ (which is endowed with the co-Zariski topology)
as the base space and $\MV$-chains as the stalks. In \cite{fg}, Filipoiu and Georgescu used $\Max(L)$ (which is endowed with  the Zariski topology) as the base space and $\local \MV$-algebras as the stalks. Di Nola, Esposito and Gerla \cite{ne} improved the methods of \cite{fg}, by choosing the stalks from given classes of $\local\MV$-algebras.  Ferraioli and Lettieri \cite{fl}, combining the techniques in \cite{dp} and \cite{fg}, got two types of sheaf representation of $\MV$-algebras. In \cite{mg} and \cite{gg}, Gehrke, Gool and Marra provided a general framework for previously known results on sheaf representations of $\MV$-algebras as \cite{dp} and \cite{fg}, through the lens of Stone-Priestley duality, using canonical extensions as an essential tool. 
In terms of the sheaf representations of $\BL$-algebras, Di Nola and Leu\c{s}tean, who adopted $\Spec(L)(\Max(L))$ as the base space and $\BL$-algebras ($\local \BL$-algebras) as the stalks, obtained the sheaf representation and the compact representation of $\BL$-algebras \cite{ne,l}. 

In \cite{mg}, Gehrke and Gool dealt with sheaf representations of a $\mathcal{V}$-algebra. In their definition  Def~3.1 \cite{mg}, they required that the $\mathcal{V}$-algebra $A$ is isomorphic to $FY$, where $F$ is a sheaf, and $FY$ is the algebra of global sections of $F$. In this paper, we loosen the  isomorphism condition in \cite{mg}
and the requirement on the stalks in \cite{fl} and further extend above results to the more general structures; namely, define the sheaf spaces of residuated lattices and obtain the sheaf representation of residuated lattices.


\section{Preliminaries}

In this section, we recall some basic notions and results to be used in the sequel.

\begin{definition}(\cite{wd}) A residuated lattice is an algebra $(L, \wedge, \vee, \otimes, \rightarrow, 0,1)$ satisfying the following conditions:
\begin{enumerate}	
	\item[(1)] $(L, \wedge, \vee, 0,1)$ is a bounded lattice;
	
	\item[(2)] $(L, \otimes,1)$ is a commutative monoid with identity 1;
	
	\item[(3)]  for any $x, y, z\in L$, $x\otimes y \leq z$ iff $x \leq y \rightarrow z$.
	\end{enumerate}
\end{definition}
\setlength{\parskip}{0.3\baselineskip}

In the following, we shall use the shorthand $L$ to denote $(L, \wedge, \vee, \otimes, \rightarrow, 0,1)$.

\begin{definition}(\cite{h,tu})  A nonempty subset $F$ of a residuated lattice $L$ is called a filter  if
	\begin{enumerate}
	\item[(1)] $x\in F$ and $x\leq y$ imply $y\in F$;
	
	\item[(2)] $x\in F$ and $y\in F$ imply $x\otimes y\in F$.
	\end{enumerate}
\end{definition}

\begin{theorem} (\cite{h,tu})  A nonempty subset $F$ of a residuated lattice $L$ is a filter if
\begin{enumerate}	
\item[(1)] $1\in F$;

\item[(2)] $x\in F$ and $x\rightarrow y\in F$ imply $y\in F$.
\end{enumerate}
\end{theorem}

\begin{remark} A filter is \emph{proper} if $F\neq L$. We will use $\mathcal{F}(L)$ to denote the set of all filters of a residuated lattice $L$.
Note that $\{1\}$ and $L$ are filters. For any  $X\subseteq L$, the filter of $L$ generated by $X$ (the smallest filter containing $X$) will be denoted by 
\textless 
$X$\textgreater 
.
In particular, the filter generated by $\{a\}$ will be denoted by 
\textless 
$a$\textgreater 
. 
To each filter $F$, we can
associate a congruence relation $\equiv_{F}$ on $L$ given by 
\begin{center}
$x\equiv_{F} y$ iff $((x\rightarrow y) \otimes (y\rightarrow x))\in F.$
\end{center}
Let $L/F$ denote the set of the congruence
classes of $\equiv_{F}$, i.e., $L/F=\{x/F|x\in L\}$, where

\noindent $x/F:=\{y\in L|y \equiv_{F} x\}.$ Define operations on $L/F$ as
follows:
\begin{center}
$x/F \sqcap y/F=(x\wedge y)/F$, $x/F \sqcup y/F =(x\vee y)/F$,

$x/F \odot y/F =(x\otimes y)/F$, $x/F \rightharpoonup y/F =(x\rightarrow y)/F$.
\end{center}
\end{remark}

\begin{remark} (\cite{gi,go}) It is easy to show that if $\{F_{i}:i\in I\}\subseteq \mathcal{F}(L)$ is a directed family ($\forall i_{1}, i_{2}\in I, \exists i_{s}\in I$ such that $F_{i_{1}}\subseteq F_{i_{s}}$ and $F_{i_{2}}\subseteq F_{i_{s}}$), then $\bigcup_{i\in I} F_{i}\in \mathcal{F}(L)$. Thus, for any $x\in L,$ \textless 
	$x$\textgreater 
	$\ll$\textless 
	$x$\textgreater 
~holds in the complete lattice $(\mathcal{F}(L), \subseteq)$ (see \cite{gi,go} for the definition of the way below relation $\ll$). It's not clear whether $F\in \mathcal{F}(L)$
and $F\ll F$ implies $F=$
\textless 
$x$\textgreater 
~for some $x\in L$.
 \end{remark}

\begin{lemma}(\cite{cj}) Let $F$ be a filter of a residuated lattice $L$. Then
	($L/F,\sqcap,\sqcup,\odot,\rightharpoonup,0/F,1/F)$ is a
	residuated lattice.
\end{lemma}

\begin{remark}(\cite{cj}) The following properties hold on any residuated lattice:
	\begin{enumerate}
	\item[(1)] $x\vee (y\otimes z)\geq (x\vee y)\otimes (x\vee z)$;
	
	\item[(2)]  \textless 
	$x$\textgreater 
	$\cap$\textless 
	$y$\textgreater 
	 $=$
	 \textless 
	$x\vee y$\textgreater .
\end{enumerate}
\end{remark}

\begin{definition}(\cite{cj}) A proper filter $P$ of a residuated lattice $L$ is called a prime filter if for any $x,y\in L$,
	$x\vee y\in P$ implies $x\in P$ or $y\in P$.
\end{definition} 

The set of all prime filters of a residuated lattice $L$ is
called the  \emph{prime spectrum} of $L$ and is denoted by $\Spec(L)$.

\begin{lemma} Let $P$ be a proper filter of a residuated lattice $L$. Then $P$ is a prime filter iff
	$F_{1} \cap F_{2} \subseteq P$ implies $F_{1}\subseteq P$ or $F_{2} \subseteq P$ for any $F_{1}, F_{2} \in \mathcal{F}(L)$.
\end{lemma}

\begin{proof} Assume that $F_{1}\cap F_{2}\subseteq P$ with $F_{1}\nsubseteq P$ and $F_{2}\nsubseteq P$. Then
	there exist $x\in F_{1}, y\in F_{2}$ such that $x\notin P$, $y\notin P$. Thus $x\vee y\in F_{1}\cap F_{2}$ and $x\vee y\notin P$, since
	$P$ is a prime filter. This shows that $F_{1}\cap F_{2}\nsubseteq P$, a contradiction.
	Conversely, assume that $x\vee y \in P$ with $x \notin P$ and $y \notin P$. Thus \textless 
	$x$\textgreater 
	$\nsubseteq P$ and \textless 
	$y$\textgreater 
	$\nsubseteq P$. Therefore 
	\textless 
	$x$\textgreater 
	$\cap$\textless 
	$y$\textgreater 
 $\nsubseteq P$. We have 
 \textless 
 $x\vee y$\textgreater 
 $\nsubseteq P$, that is, $x\vee y \notin P$, again a contradiction.
	
	Next, for any $X\subseteq L$, we will write $D(X)=\{P\in$ Spec$(L)| X\nsubseteq P\}$.
	For any $a\in L$, $D(\{a\})$ shall be denoted simply by $D(a)$.
\end{proof} 

\begin{lemma} Let $L$ be a residuated lattice. Then
	\begin{enumerate}
	\item[(1)] $X\subseteq Y\subseteq L$ implies $D(X)\subseteq D(Y)$;
	
	\item[(2)] $D(X)=D$(\textless 
	$X$\textgreater ).
\end{enumerate}
\end{lemma}

\begin{proof} (1) is trivial.
	
	(2) Since $X\subseteq$
	 	\textless 
	 $X$\textgreater 
	 , from (1), we have that $D(X)\subseteq D$(\textless 
	$X$\textgreater 
	 ). Conversely, suppose $P\in D$(\textless 
$X$\textgreater 
), then \textless 
$X$\textgreater 
$\nsubseteq P$. It follows, by the definition of \textless 
$X$\textgreater 
, that $X \nsubseteq P$. That is, $P\in D(X)$. Thus, we have
	$D(X)=$
	D(\textless 
	$X$\textgreater 
	).
\end{proof}

We now recall some basic notions about topology to be used later. For more about these, we refer to \cite{k}.

A topological space is a pair $(X, \tau)$, where $X$ is a nonempty set and $\tau$ is a family of subsets of $X$, called the topology, such that (i) $\emptyset, X \in \tau$, (ii) a finite intersection of members of $\tau$ is in $\tau$ and (iii) an arbitrary union of members of $\tau$ is in $\tau$.

The members of $\tau$ are called \emph{open sets} of $X$ and the elements of $X$ are called \emph{points}. A \emph{neighbourhood of a point $x$} in a topological space $X$ is a subset $W\subseteq X$ such that there exists an open set $U$ of $X$ satisfying $x\in U\subseteq W$. A set $U$ is open iff $U$ is the neighbourhood of every $x\in U$. A \emph{base $\mathcal{B}$} for a topology $\tau$ is a collection of open sets in $\tau$ such that  every open set in $\tau$
is a union of some members of $\mathcal{B}$.

\begin{lemma} A collection $\mathcal{B}$ of subsets of set $X$ is the base for some topology iff $X=\bigcup\{V: V\in \mathcal{B}\}$ and
	if $V_{1}, V_{2}\in \mathcal{B}, x\in V_{1}\cap V_{2}$, then there exists $V\in \mathcal{B}$ such that $x\in V\subseteq V_{1}\cap V_{2}$.
\end{lemma}

A function $f: X\longrightarrow Y$ from a topological space $(X, \tau)$ to a topological space $(Y, \sigma)$ is \emph{continuous at a point $x\in X$} if for any neighbourhood $V$ of $f(x)$, there is a neighbourhood $U$ of $x$ such that $f(U)\subseteq V$. The function is called \emph{continuous} if it is continuous everywhere. For any function $f: X\longrightarrow Y$ between two topological spaces, $f$ is continuous iff for any open set $W$ of $Y$, $f^{-1}(W)$ is open in $X$ iff for any open set $V$ in a base $\mathcal{B}$ of $Y$, $f^{-1}(V)$ is open in $X$. A function $f:X\longrightarrow Y$ between two topological spaces $X$ and $Y$ is an \emph{open function} if for any open set $U$ of $X$, $f(U)$ is an open set of $Y$. A function $f:X\longrightarrow Y$ between two topological spaces $X$ and $Y$ is an open function iff for any open set $W$ in a base of $X$, $f(W)$ is open in $Y$. A bijective function $f: X\longrightarrow Y$ between two topological spaces is a \emph{homeomorphism} if both $f$ and $f^{-1}$ are continuous. A bijective function $f: X\longrightarrow Y$ between two topological spaces is a homeomorphism iff $f$ is continuous and open.

\begin{theorem} For any residuated lattice $L$, the family $\{D(X)|X\subseteq L\}$
	is a topology on Spec($L$), which we call the \emph{Stone topology} on $L$.
\end{theorem}
\begin{proof}
	We complete the proof by verifying each of the following.
	
	(1) $D(L)= \Spec(L)$ and $D(1)=\emptyset$.
	
	(2) For any $X\subseteq L$ and $Y\subseteq L$, $D(X)\bigcap D(Y)=D$(\textless 
	X\textgreater 
		$\cap$\textless 
	Y\textgreater 
	).
	
	(3) For any family $\{X_{i}| i\in I\}$ of subsets of $L$, $D(\bigcup_{i\in I} X_{i})=\bigcup_{i\in I} D(X_{i})$.
	
	For any $P\in \Spec(L), L\nsubseteq P$. Thus $D(L)= \Spec(L)$.
	For any $P\in \Spec(L), \{1\}\subseteq P$. Hence $P\notin D(1)$. Therefore $D(1)=\emptyset$. Thus (1) holds.
	
	Since  \textless 
	X\textgreater 
	$\cap$\textless 
	Y\textgreater 
	 $\subseteq$\textless 
	 X\textgreater 
	 ,\textless 
	 Y\textgreater 
	 ,
	 by Lemma~2.10, we have $D$(\textless 
	 X\textgreater 
	 $\cap$\textless 
	 Y\textgreater 
	 )
	 $\subseteq D$(\textless 
	 $X$\textgreater 
	  ) $\bigcap D$(\textless 
	  $Y$\textgreater 
	 )
$=D(X)\bigcap D(Y)$. Conversely, suppose that $P\in D(X)\bigcap D(Y)$, then $X\nsubseteq P$ and $Y\nsubseteq P$. Hence 
	 \textless 
	 $X$\textgreater 
	 $\nsubseteq P$ 
	 and \textless 
	 $Y$\textgreater 
$\nsubseteq P$. 
	 By Lemma~2.9, we have \textless 
	 $X$\textgreater 
	 $\cap$\textless 
	 $Y$\textgreater 
	 $\nsubseteq P$. This shows that $P\in D($\textless 
	 X\textgreater 
	 $\cap$\textless 
	 Y\textgreater 
	 ). Therefore $D(X)\bigcap D(Y)=D$(\textless 
	 $X$\textgreater 
	 )
	 $\bigcap D$(\textless 
	 $Y$\textgreater 
	 )$\subseteq D$(\textless 
	 $X$\textgreater 
	 $\cap$\textless 
	 $Y$\textgreater 
	 ). Hence (2) holds.
	
	Lastly, we verify (3). Suppose that $P\in D(\bigcup_{i\in I} X_{i})$, then there exists $i\in I$ such that $X_{i}\nsubseteq P$. Thus we have $P\in D(X_{i})\subseteq \bigcup_{i\in I} D(X_{i})$. Hence $D(\bigcup_{i\in I} X_{i}) \subseteq \bigcup_{i\in I} D(X_{i})$. The reverse inclusion holds by Lemma~2.10 (1).
\end{proof}

\begin{remark} By Lemma~2.10 (2) and Theorem~2.12, we know that the open sets in the spectrum $\Spec(L)$ are exactly the subsets in $\{D(F):F\in \mathcal{F}(L)\}$.
\end{remark}

\begin{theorem} For any residuated lattice $L$, the family $\{D(a)\}_{a\in L}$ is a base for the \emph{Stone topology} on $\Spec(L)$.
\end{theorem}
\begin{proof} Suppose that $X\subseteq L$ and $D(X)$ is an arbitrary open set of $\Spec(L)$, then $D(X)=D(\bigcup_{a\in X}\{a\})=\bigcup_{a\in X} D(a)$. Hence every open set $U$ of $\Spec(L)$ is the union of a subset of $\{D(a)\}_{a\in U}$.
\end{proof}

\begin{proposition} For any $P\in \Spec(L)$, $O(P)$ is a proper filter of a residuated lattice $L$ satisfying $O(P)\subseteq P$, where
	$O(P)=\{x\in L|a\vee x=1$ for some $a\in L-P\}$.
\end{proposition}

\begin{proof} Since $1\notin L-P$, it follows immediately that $0\notin O(P)$. If $x\in O(P)$ and $x\leq y$, then there exists $a\in L-P$ such that $a\vee x=1$. Hence $1=x\vee a\leq y\vee a$. Therefore $y\vee a=1$, showing
	that $y\in O(P)$. Next, if  $x, y\in O(P)$, then there exist $a,b\in L-P$  such that  $a\vee x=1$ and $b\vee y=1$. So $a\vee b\in L-P$,
	because $P$ is a prime filter of $L$. Thus $(a\vee b) \vee (x\otimes y) \geq (a\vee b\vee x)\otimes (a\vee b\vee y)=1\otimes 1=1$. Therefore
	$(a\vee b) \vee (x\otimes y)=1$. This shows that $x\otimes y\in O(P)$. For any $x\in O(P)$, there exists $a\in L-P$  such that  $a\vee x=1\in P$.
	Thus $x\in P$.
\end{proof}

\begin{example} Let $L=\{0, a, b, c, 1\}$ with $ 0 < a, b < c < 1$ and $a, b$ incomparable. The operations $\otimes$ and
	$\rightarrow$ are defined as follows:\smallskip
\begin{center}
\begin{minipage}[c]{0.2\textwidth}
\centering
		\begin{tabular}{|c|c|c|c|c|c|}
			\hline  $\otimes$  & 0 & $a$ & $b$ & $c$ & 1 \\
			\hline 	0 & 0 & 0 & 0 & 0 & 0\\
			\hline  $a$ & 0 & $a$ & 0 & $a$ & $a$\\
			\hline  $b$ & 0 & 0 & $b$ & $b$ & $b$\\
			\hline  $c$ & 0 & $a$ & $b$ & $c$ & $c$\\
			\hline  1 & 0 & $a$ & $b$ & $c$ & 1\\
			\hline
		\end{tabular}\label{mytable}
	\end{minipage}
\begin{minipage}[c]{0.5\textwidth}
\centering
\begin{tabular}{|c|c|c|c|c|c|}
			\hline  $\rightarrow$  & 0 & $a$ & $b$ & $c$& 1 \\
			\hline 	0 & 1 & 1 & 1 & 1 & 1\\
			\hline  $a$ & $b$ & 1 & $b$ & 1 & 1\\
			\hline  $b$ & $a$ & $a$ & 1 & 1 & 1\\
			\hline  $c$ & 0 & $a$ & $b$ & 1 & 1\\
			\hline  1 & 0 & $a$ & $b$ & $c$ & 1\\
			\hline
		\end{tabular}
	\end{minipage}
\end{center}\smallskip
\noindent Then $L$ becomes a residuated lattice (see \cite{cj}). The filters of $L$ are $\{1\}, \{c, 1\}$, $\{a, c, 1\}, \{b, c, 1\}$ and $L$. It is
easy to check that the prime filters of $L$ are $\{a, c, 1\}, \{b, c, 1\}$, and $O(\{a, c, 1\})=\{1\}$, $O(\{b, c, 1\})=\{1\}$.
\end{example}

\begin{example} Let $L=\{0, a, b, 1\}$ with $ 0 < a, b < 1$ and $a, b$ incomparable. The operations $\otimes$ and
	$\rightarrow$ are defined as follows:\smallskip
\begin{center} 
\begin{minipage}[c]{0.2\textwidth}
	\centering
	\begin{tabular}{|c|c|c|c|c|}
		\hline  $\otimes$  & 0 & $a$ & $b$ & 1 \\
		\hline 	0 & 0 & 0 & 0 &  0\\
		\hline  $a$ & 0 & $a$ & 0 &  $a$\\
		\hline  $b$ & 0 & 0 & $b$ &  $b$\\
    	\hline  1 & 0 & $a$ & $b$ &  1\\
		\hline
	\end{tabular}\label{mytable}
\end{minipage}
\begin{minipage}[c]{0.5\textwidth}
	\centering
	\begin{tabular}{|c|c|c|c|c|}
		\hline  $\rightarrow$  & 0 & $a$ & $b$ & 1 \\
		\hline 	0 & 1 & 1 & 1 &  1\\
		\hline  $a$ & $b$ & 1 & $b$ &  1\\
		\hline  $b$ & $a$ & $a$ & 1 &  1\\
	    \hline  1 & 0 & $a$ & $b$ &  1\\
		\hline
	\end{tabular}
\end{minipage}
\end{center}\smallskip
\noindent It is routine to verify that with the above operations, $L$ is a residuated lattice and the filters of $L$ are $\{1\}, \{a, 1\}, \{b,  1\}$ and $L$. In addition, the prime filters of $L$ are $\{a, 1\}, \{b, 1\}$ and $O(\{a, 1\})=\{a, 1\}$, $O(\{b, 1\})=\{b, 1\}$.
\end{example}

\section{The sheaf representations of residuated lattices}

In this section, we introduce the notion of sheaf space of residuated lattices and construct the sheaf representations of residuated lattices.

\begin{definition} A sheaf space of residuated lattices is a triple $(E, p, X)$ satisfying the following conditions:
	\begin{enumerate}
	\item[(1)] Both $E$ and $X$ are topological spaces.
	
	\item[(2)] $p:E\longrightarrow X$ is a local homeomorphism from $E$ onto $X$, i.e. for any $e\in E$, there are open neighbourhoods $U$ and
	$U^{\prime}$ of $e$ and $p(e)$  such that  $p$ maps $U$ homeomorphically onto $U^{\prime}$.
	
	\item[(3)] For any $x\in X, p^{-1}(\{x\})=E_{x}$ is a residuated lattice.
	
	\item[(4)] The functions defined by $(a,b)\longmapsto a\wedge_{x} b, (a,b)\longmapsto a\vee_{x} b, (a,b)\longmapsto a\otimes_{x} b, (a,b)\longmapsto a\rightarrow_{x} b$ from the set $\{(a,b)\in E \times E | p(a)=p(b)\}$ into $E$ are continuous, where $x=p(a)=p(b)$.
	
	\item[(5)] The functions $\underline{0},\underline{1}: X\longrightarrow E$ assigning to every $x$ in $X$ the $0_{x}$ and $1_{x}$ of $E_{x}$ respectively, are continuous.
	\end{enumerate}
\end{definition}

\begin{remark} In the Definition~3.1, $E$ is usually called the total space, $X$ as the base space and $E_{x}$ is called the stalk of $E$
	at $x\in X$.
\end{remark}

\begin{definition} Let $(E, p, X)$ be a sheaf space of residuated lattices. For any $Y\subseteq X$, a function $\sigma: Y\longrightarrow E$ is called a section over $Y$ if it is continuous  such that for any $y\in Y, p(\sigma(y))=y$.
\end{definition}

\begin{remark} If we define the operations pointwisely on the set of all sections over $Y$, it constitutes a residuated lattice. We denote it by $\Gamma(Y, E)$. The elements of $\Gamma(X, E)$ are called \emph{global sections}.
\end{remark}

\begin{definition}(\cite{tu}) Suppose that $L$ and $L^{\prime}$ are residuated lattices. A residuated lattice morphism is a function $h: L\longrightarrow L^{\prime}$ such that $h(a \wedge_{L} b) = h(a) \wedge_{L^{\prime}} h(b), h(a \vee_{L} b) = h(a) \vee_{L^{\prime}} h(b),
	h(a \otimes_{L} b) = h(a) \otimes_{L^{\prime}} h(b), h(a \rightarrow_{L} b) = h(a) \rightarrow_{L^{\prime}} h(b)$ and $h(0) = 0^{\prime}, h(1) = 1^{\prime}$.
\end{definition}

\begin{definition} A sheaf representation of a residuated lattice $L$ will mean an injective residuated lattice morphism $\phi :L \longrightarrow \Gamma(X, E)$ from $L$ to the residuated lattice $\Gamma(X, E)$ of global sections of a sheaf space of residuated lattices $(E, p, X)$.
\end{definition}

\begin{lemma} Let $(E, p, X)$ be a sheaf space of residuated lattices. If we define $p_{x}:\Gamma (X,E)\longrightarrow E_{x}$
	by $p_{x}(\sigma)=\sigma(x)$, then for any $x\in X,$ $p_{x}$ is a residuated lattice morphism.
\end{lemma}

\begin{proof} Here we only prove that $\otimes$ is a morphism, the proofs for other operations are similar. For any $\sigma, \mu\in \Gamma(X, E)$,
	$p_{x}(\mu\otimes \sigma)=(\mu\otimes \sigma)(x)=\mu(x)\otimes_{x} \sigma(x)=p_{x}(\mu)\otimes_{x} p_{x}(\sigma)$,
	and $p_{x}(\underline{0})=\underline{0}(x)=0_{x}, p_{x}(\underline{1})=\underline{1}(x)=1_{x}.$
\end{proof}

\begin{lemma} If $L$ is a residuated lattice, then for any $a\in L, V(a)=\{P\in  \Spec(L)|a\in O(P)\}$ is open in $\Spec(L)$.
\end{lemma}

\begin{proof} Assume that $P\in V(a)$, then $a\in O(P)$. Thus there exists $b\in L-P$  such that  $a\vee b=1$. If $Q\in D(b)$, then
	$b\notin Q$. Hence $a\in O(Q)$, i.e. $Q\in V(a)$. Therefore $P\in D(b)\subseteq V(a)$. This shows that $V(a)$ is open.
\end{proof}

In the sequel, we will construct a sheaf space for each residuated lattice using the residuated lattice $L$ and the topological
space $\Spec(L)$. Let $E_{L}$ be the disjoint union of the set $\{L/O(P)\}_{P\in \Spec(L)}$ and $\pi: E_{L}\longrightarrow \Spec(L)$ the canonical projection.

\begin{theorem} Let $L$ be a residuated lattice. Then the family $\mathcal{B}=\{D(F, a):F\in \mathcal{F}(L)$ and $a\in L\}$
	is a base for a topology on $E_{L}$, where $D(F, a)=\{a_{P}:P\in D(F)\}$ and $a_{P}=a/O(P)$.
\end{theorem}

\begin{proof} We complete the proof  in two steps.
	
	(i) For every $D_{1}, D_{2}\in \mathcal{B}$ and $x\in D_{1}\cap D_{2}$, there exists a $D\in \mathcal{B}$  such that
	$x\in D \subseteq D_{1}\cap D_{2}$.
	
	Take $D_{1}= D(F_{1}, a), D_{2}=D(F_{2}, b)$ with $F_{1}, F_{2}\in \mathcal{F}(L)$ and $a,b\in L$. Suppose that $x\in D(F_{1}, a),$ 
	$x\in D(F_{2}, b)$, then there exists $P\in D(F_{1})$ and $Q\in D(F_{2})$ such that  $x=a_{P}=b_{Q}$. Thus $P=Q$ and $(a\rightarrow b)\otimes (b\rightarrow a)\in O(P)$. Hence $P\in D(F_{1})\cap D(F_{2})\cap V((a\rightarrow b)\otimes (b\rightarrow a)):=W$. By Remark~2.13 and Lemma~3.8, we have that $W$ is open in Spec$(L)$. Hence there exists a filter $F$  such that  $P\in D(F)\subseteq W\subseteq D(F_{1})\cap D(F_{2}).$ Therefore $D(F, a)=\{a_{p}: P\in D(F)\}\subseteq D(F_{1}, a)$ and $D(F, a)=\{a_{P}: P\in D(F)\}$
	$=\{b_{P}:P\in D(F)\}\subseteq \{b_{P}: P\in D(F_{2})\}=D(F_{2}, b)$, because $(a\rightarrow b)\otimes (b\rightarrow a)\in O(P)$. Therefore $x\in D(F, a)\subseteq D(F_{1}, a)\cap D(F_{2}, b)$.
	
	(ii) For every $x\in E_{L}$, there exists a $D\in \mathcal{B}$ with  $x\in D$.
	
	Suppose that $x\in E_{L}$, then there exists $a\in L$ and $P\in \Spec(L)$  such that  $x=a_{P}$. Thus there exists $G\in \mathcal{F}(L)$ such that $P\in D(G)$. This shows that $x\in D(G, a)$.
\end{proof}

In the sequel, we will use $\mathcal{T}(\mathcal{B})$ to denote the topology on $E_{L}$ generated by the above $\mathcal{B}$.

\begin{theorem} The assignment $\pi: E_{L}\longrightarrow \Spec(L)$ defined by $a_{P}\longmapsto P$ is a local homeomorphism of $(E_{L}, \mathcal{T}(\mathcal{B}))$ onto $\Spec(L)$.
\end{theorem}

\begin{proof} The mapping $\pi$ is well defined and it is clear that $\pi$ is surjective.
	Suppose that $a_{P}\in E_{L}$ and $U=D(F, a)$ is an open neighbourhood of $a_{P}$ from $\mathcal{B}$. Obviously,
	$\pi(D(F, a))=D(F)$. The restriction $\pi_{U}$ of $\pi$ to $U$ is injective from $U$ into $D(F)$.
	
	(i) $\pi_{U}$ is continuous:
		In fact, suppose that $D(G)$ is an open set of $\Spec(L)$, then $D(F)\cap D(G)=D(F\cap G)$
	is a base open set in $D(F)$. Also $\pi^{-1}_{U}(D(F\cap G))=\{a_{P}: P\in D(F\cap G)\}=D(F\cap G, a)$ and it is an open
	subset of $D(F, a)$.
	
	(ii) $\pi_{U}$ is open:	
	To see this, assume that $D(H, b)$ is a base open set of $E_{L}$. Then $D(H, b)\cap U$ is a base open subset of $U$. Also $\pi_{U}(U\cap D(H, b))=D(F)\cap D(H)$, which is open in $D(F)$.
\end{proof}

\begin{proposition} For any $a\in L$, the function $\hat{a}: \Spec(L) \longrightarrow E_{L}$ defined by $\hat{a}(P)=a_{P}$ is a global section of $(E_{L}, \pi, \Spec(L))$.
\end{proposition}

\begin{proof} First, $\pi(\hat{a}(P))=\pi(a_{P})=P$. Next we prove that $\hat{a}$ is continuous. Actually, for any $D(F, a)\in \mathcal{B},$ 
	${\hat{a}}^{-1}(D(F, a))=D(F)$, which is open in $\Spec(L)$. And for any $b\in L$, $b\neq a, D(F, b)\in \mathcal{B},$
	\begin{equation}
		\begin{aligned}
			{\hat{a}}^{-1}(D(F, b)) &=D(F)\bigcap \{P|a_{P}=b_{P}\}\\
			&=D(F)\bigcap \{P\in \Spec(L)|(a\rightarrow b)\otimes (b\rightarrow a)\in O(P)\}\\
			&=D(F)\bigcap  \{P\in \Spec(L)|a\rightarrow b\in O(P)\} \bigcap  \{P\in \Spec(L)|b\rightarrow a\in O(P)\}\\
			&=D(F) \bigcap V(a\rightarrow b)\bigcap V(b\rightarrow a)\\
		\end{aligned}
	\end{equation}	
\noindent By Remark~2.13 and Lemma~3.8, we know that ${\hat{a}}^{-1}(D(F, b))$ is open in $\Spec(L)$. 
\end{proof}
	
\begin{corollary} The functions $\hat{0}: \Spec(L)\longrightarrow E_{L}$ and $\hat{1}: \Spec(L)\longrightarrow E_{L}$ are global sections of $(E_{L}, \pi, \Spec(L))$.
\end{corollary}

Let $E_{L}\vartriangle E_{L}=\bigcup\{E_{P}\times E_{P}: P\in$ Spec$(L)\}$ and equip  $E_{L}\vartriangle E_{L}$ with the subspace topology of the product space $E_{L}\times E_{L}$. It is well known that a base for the topology on $E_{L}\times E_{L}$ is $\mathcal{B}^{\prime}=\{D(F, a)\times D(G, b): F, G\in \mathcal{F}(L)$ and $a,b\in L\}$. Thus a base for the induced topology on $E_{L}\vartriangle E_{L}$ is given by $\mathcal{B}^{\prime\prime}=\{(B(a, b), F): F\in \mathcal{F}(L)$ and $a,b\in L\}$, where $(B(a, b), F)$ is the set $\{(a_{P}, b_{P}): P\in D(F)\}$.

\begin{proposition} For any $P\in \Spec(L)$, the functions $(a_{P}, b_{P})\longmapsto a_{P}\wedge_{P} b_{P}, (a_{P}, b_{P})\longmapsto a_{P}\vee_{P} b_{P}$, 
$ (a_{P}, b_{P})\longmapsto a_{P}\otimes_{P} b_{P}, (a_{P}, b_{P})\longmapsto a_{P}\rightarrow_{P} b_{P}$ from the set
	$\{(a_{P}, b_{P})\in E_{L}\times E_{L} |\pi(a)=\pi(b)\}$ into $E_{L}$ are continuous, where $P=\pi(a)=\pi(b)$.
\end{proposition}

\begin{proof} We only prove the continuity of the operation $\otimes_{P}$. The proofs for the rest of the operations are similar.
	Let $(a_{P}, b_{P})\in E_{L}\vartriangle E_{L}$ and $D(F, a\otimes b)$ a neighbourhood of $(a\otimes_{P} b)_{P}$. Then $(B(a, b), F)$
	is a neighbourhood of $(a_{P}, b_{P})$, whose image by $\otimes_{P}$ is contained in $D(F, a\otimes b)$.
\end{proof}

\begin{theorem} For any residuated lattice $L$, $(E_{L}, \pi, \Spec(L))$ is a sheaf space of $L$.
\end{theorem}

\begin{proof} For any $P\in \Spec(L)$, $\pi^{-1}(\{P\})=L/O(P)$. And for any $P\in \Spec(L)$, $O(P)$ is a proper
	filter of $L$, thus $L/O(P)$ is a residuated lattice. By Theorem~3.10, Proposition~3.11, Corollary~3.12 and Proposition~3.13, we deduce that  $(E_{L}, \pi, \Spec(L))$ is a sheaf space of $L$.
\end{proof}

\begin{lemma}(\cite{cj}) If $F$ is a filter of a residuated lattices $L$ and $a\in L-F$, then there exists a prime filter $P$ of $L$  such that  $F\subseteq P$ and $a\notin P$.
\end{lemma}

\begin{proposition} $\bigcap \{P | P\in \Spec(L)\}=\{1\}$.
\end{proposition}

\begin{proof} Clearly $\{1\}\subseteq\bigcap \{P | P\in \Spec(L)\}$. Conversely assume that $a\neq 1$, then by Lemma~3.15, there is a $P\in \Spec(L)$  such that  $a\notin P$. Thus $a\notin \bigcap \{P | P\in \Spec(L)\}$. Therefore $\bigcap \{P | P\in \Spec(L)\}\subseteq \{1\}$.
\end{proof}

For any $P\in \Spec(L)$, $O(P)$ is a subset of $P$ and $1\in O(P)$, thus the result below follows immediately.

\begin{corollary} $\bigcap \{O(P) | P\in \Spec(L)\}=\{1\}$.
\end{corollary}

\begin{theorem} If $L$ is a residuated lattice, then the family $\{O(P)\}_{P\in \Spec(L)}$ canonically determines a sheaf representation of $L$.
\end{theorem}

\begin{proof} Define $\varphi: L\longrightarrow \Gamma(\Spec(L), E_{L})$ by $\varphi(a)=\hat{a}$. We only prove that
	for any $a,b\in L, \varphi(a\otimes b)=\varphi(a)\otimes_{P} \varphi(b)$. The proofs for rest of the operations are similar.
	For any $P\in \Spec(L)$, $\varphi(a\otimes b)(P)=(\widehat{a\otimes b})(P)=a\otimes b/O(P)=a/O(P)\otimes_{P} b/O(P)=\hat{a}(P)\otimes_{P}
	\hat{b}(P)=\varphi(a)(P)\otimes_{P} \varphi(b)(P)=(\varphi(a)\otimes_{P} \varphi(b))(P)$. Thus $\varphi(a\otimes b)=\varphi(a)\otimes_{P} \varphi(b)$.
Next, we prove that the mapping $\varphi$ is injective. Assume that $\varphi(a)=\varphi(b)$. Then for any $P\in \Spec(L)$, $a_{P}=b_{P}$. Thus $(a\rightarrow b)\otimes (b\rightarrow a)\in \bigcap \{O(P) | P\in \Spec(L)\}=\{1\}$, i.e. $a=b$.
\end{proof}

\begin{problem}
For what $L$, is the mapping $\varphi$ surjective?	
\end{problem}

\section{Conclusions and future work}

~~~~~In this paper, we investigate the properties of the family of all the prime filters of residuated lattices. Based on this, we construct the sheaf space of residuated lattices and obtain a sheaf representation of residuated lattices.

In \cite{fl}, Ferraioli and Lettieri took the primary ideals as the corresponding ideals of the prime ideals and proved that every $\MV$-algebra and the $\MV$-algebra of all global sections of its sheaf space are isomorphic. In \cite{mg,gg},  the scholars proved every $\MV$-algebra $A$ is isomorphic to the $\MV$-algebra of global sections of a sheaf $F$ of $\MV$-algebras with stalks that are linear. In our future work, we will investigate when these results hold for a residuated lattice, specifically, for what residuated lattice $L$, the mapping $\varphi: L\longrightarrow \Gamma(\Spec(L), E_{L})$ is surjective. For example, is $\varphi$ surjective for any Heyting algebra $L$?

\bibliographystyle{./entics}

\end{document}